\documentclass[10pt]{article}

\usepackage{times}
\usepackage{amsthm}
\usepackage{amsmath}
\usepackage{amssymb}
\usepackage{mathrsfs}

\def\A{{\mathcal A}}

\def\th{\theta}
\def\Th{\Theta}
\def\N{\mathbb{N}}

\def\AA{\mathfrak A}

\def\H{\mathcal H}
\def\K{\mathcal K}

\def\p{\parallel}

\def\S{\mathcal S}

\def\Ai{\mathcal A^\infty}
\def\AAi{\mathfrak A^\infty}

\def\si{\sigma}
\def\Si{\Sigma}

\def\<{\langle}
\def\>{\rangle}

\providecommand{\CC}{\mathfrak{C}}

\def\N{\mathbb{N}}

\def\[{\left[}
\def\]{\right]}
\def\<{\left<}
\def\>{\right>}
\def\({\left(}
\def\){\right)}

\def\H{\mathcal H}
\def\K{\mathcal K}

\def\S{\mathscr S}

\def\CC{{\mathfrak C}}

\def\<{\langle}
\def\>{\rangle}

\providecommand{\CC}{\mathfrak{C}}

\newtheorem{theorem}{Theorem}[section] 
    
\newtheorem{corollary}[theorem]{Corollary}
\newtheorem{proposition}[theorem]{Proposition}
\newtheorem{remark}[theorem]{Remark}
\newtheorem{example}[theorem]{Example}
\newtheorem{definition}[theorem]{Definition}

\begin{document}

\title{Modulation Spaces and Representations\\ for Rieffel's Quantization}

\author{Marius M\u antoiu  }

\date{\small}
\maketitle\vspace{-1cm}

\bigskip
\medskip
Departamento de Matem\'aticas, Universidad de Chile,

Las Palmeras 3425, Casilla 653, Santiago, Chile

\emph{E-mail:} mantoiu@uchile.cl

\bigskip

\begin{abstract}
We define localized modulation maps and modulation spaces of symbols suited to the study of Rieffel's deformation  quantization pseudodifferential calculus. They are used to generate Hilbert space representations for the quantized $C^*$-algebras, starting from covariant representations of the corresponding twisted $C^*$-dynamical system. In the case of an Abelian undeformed algebra, orthogonal relations and extra information about the representations are obtained.
\end{abstract}

%.......................................................................................................
\section*{Introduction}\label{duci}
%.......................................................................................................

In a famous paper \cite{Rie1}, Mark Rieffel introduced a general deformation quantization procedure starting with the action $\Th$ of a symplectic space $\Xi$ on a $C^*$-algebra $\A$\,, commutative or not. The outcome is another $C^*$-algebra $\AA\,,$ also endowed with an action of the symplectic space. The construction is functorial, has many applications and plays an important role both in $C^*$-algebra and in pseudodifferential theory, being a generalization of the standard Weyl calculus.

\smallskip
The symplectic form allows defining a $2$-cocycle $\kappa$ on $\Xi$\,. The same data $(\A,\Th,\Xi,\kappa)$ also permit defining a twisted crossed product $C^*$-algebra $\A\!\rtimes_\Th^\kappa\!\Xi$\,, as in \cite{PR1,PR2}. In \cite{BM} we made the connection between the two constructions, putting in evidence an isomorphism $M:\mathfrak K\otimes\AA\rightarrow\A\!\rtimes_\Th^\kappa\!\Xi$\,, where $\mathfrak K$ is an elementary $C^*$-algebra containing densely the Schwartz space $\S(\Xi)$\,. This lead to several applications. Starting from Kasprzak's approach to Rieffel's quantization \cite{Kp}, Neshveyev exhibited in \cite{Ne} a similar connection in a more general setting. In the present article, we intend to further use the mentioned isomorphism to define suitable spaces of functions for Rieffel's calculus and to explore some of its representations in Hilbert spaces.

\smallskip
If the initial $C^*$-algebra $\A$ is Abelian with Gelfand spectrum $\Sigma$\,, one can view Rieffel's formalism as a generalized version of the Weyl calculus associated to the topological dynamical system $(\Sigma,\Theta,\Xi)$ and the elements of $\A$ as generalized pseudodifferential symbols. The standard form is recovered essentially when $\Sigma=\Xi$ and $\Theta$ is the action of the vector group $\Xi$ on itself by translations. So, aside applications in Deformation Quantization and Noncommutative Geometry, one might want to use Rieffel's calculus for purposes closer to the traditional theory of pseudodifferential operators. In \cite{Ma,Ma1}, relying on the strong functorial connections between "the classical data" $(\Sigma, \Theta, \Xi)$ and the quantized algebra $\AA$\,, we used the formalism to solve several problems in spectral theory. Other potential applications are in view; their success partly depends of our ability to supply families of function spaces suited to the calculus. Since H\"ormader-type symbol spaces seem to be rather difficult to define and use, we turned our attention to the problem of adapting modulation spaces to this general context.

\smallskip
Modulation spaces and more general coorbit spaces are Banach function spaces introduced by H. Feichtinger and C. Gr\"ochenig \cite{Fe,Fe1,FG} and already useful in many fields of pure and applied mathematics. They are defined by imposing suitable norm-estimates on a certain type of transformations of the function one studies. In the standard case, these transformations involve a combination of translations and multiplications with phase factors. The literature on this topic is too vast to be reviewed here.

\smallskip
After J. Sj\"ostrand rediscovered one of these spaces in the framework of pseudodifferential operators \cite{Sj,Sj1}, the interconnection between modulation spaces and pseudodifferential theory developed considerably, as in \cite{Bu,Gr,Gr1,Gr2,GH,GR,HRT,LdG,T,To} and references therein. The modulation strategy supplies both valuable symbol spaces used for defining the pseudodifferential operators and good function spaces on which these operators apply. From several points of view, the emerging theory is simpler and sharper than that relying on "traditional function spaces". Extensions to pseudodifferential operators constructed on locally compact Abelian groups are available \cite{GS}. In \cite{MP} the case of the magnetic Weyl calculus is considered in a modulation setting, while in \cite{BB,BB4} modulation spaces are defined and studied for the magnetic Weyl calculus defined by representations of nilpotent Lie groups. In some recent publications, as \cite{FR,Ma2,MP1,RU} for instance, the theory is developed withour referring to group theory.

\smallskip
In this article we start the project of defining and using modulation spaces and Hilbert space representations adapted to Rieffel's quantization. Only general constructions will be presented here, making efforts not to exclude the case of a non-commutative "classical" algebra $\A$. Extensions, more examples and a detailed study of the emerging spaces will be presented elsewhere.

\smallskip
Sections 1 and 2 contain the basic constructions. Very roughly, the modulation strategy starts by defining linear injective maps (called modulation maps) from the smooth algebra $\AA^\infty$ to the twisted crossed product $\A\!\rtimes_\Theta^\kappa\!\Xi$\,, indexed by "windows" belonging to Schwartz space $\S(\Xi)$\,. We insist that, for self-adjoint idempotent windows, these maps should be morphisms of $^*$-algebras, at the price of deviating to a certain extent from the previous definitions, given for the usual Weyl calculus. Actually, all these morphisms are all obtained by suitably "localizing" a single isomorphism \cite{BM} sending the $^*$-algebra $\S(\Xi;\AA^\infty)$ (with the Rieffel-type structure for the doubled classical data) to $\S(\Xi;\A^\infty)$ seen as a $^*$-subalgebra of $\A\rtimes_\Theta^\kappa\Xi$\,. In addition, they extend to embeddings of $\AA$ in the twisted crossed product $\A\rtimes_\Theta^\kappa\Xi$\,. We use these modulation maps to induce norms on $\AA^\infty$ from norms defined on $\S(\Xi;\A^\infty)$\,. As a reward for our care to preserve algebraic structure, one gets in this way Banach algebra norms from Banach algebra norms, $C^*$-norms from $C^*$-norms, etc. In particular, Rieffel's algebra $\AA$ is presented as the modulation space induced from $\A\rtimes_\Theta^\kappa\Xi$\,. We also address the problem of independence of the resulting Banach spaces under the choice of the window.

\smallskip
Then, in Section 3, we turn to Hilbert space representations. By using localization with respect to idempotent windows, the $^*$-representations of the twisted crossed product $\A\!\rtimes_\Theta^\kappa\!\Xi$ (indexed by covariant representations of the twisted $C^*$-dynamical system $(\A,\Theta,\Xi,\kappa)$) automatically supply $^*$-representations of the smooth algebra $(\AA^\infty,\#)$\,, which extend to full $C^*$-representations of $(\AA,\#)$\,. This also allows one to express the norm in $\AA$ (initially defined by Hilbert module techniques) in a purely Hilbert space language.

\smallskip
In the Abelian case we have previously listed families of Schr\"odinger-type representations in the Hilbert space $L^2(\mathbb R^n)$ defined by orbits of the topological dynamical system $(\Sigma, \Theta, \Xi)$\,. They were used in [19] in the spectral analysis of Quantum Hamiltonians. We are going to show in Section 4 that their Bargmann transforms can be obtained from some canonical representations of the twisted crossed product applied to symbols defined by the modulation maps. We also prove orthogonality relations, relying on a choice of an invariant measure on $\Sigma$\,.

%.......................................................................................................
\section{Rieffel quantization and its connection with twisted crossed products}\label{sectra}
%......................................................................................................

We start with a quadruplet $\left(\A,\Th,\Xi,[\![\cdot,\cdot,]\!]\right)$ formed of a (finite dimensional real) symplectic space $(\Xi,[\![\cdot,\cdot,]\!])$ and a strongly continuous action $\Th$ of $\Xi$ by automorphisms of a $C^*$-algebra $\A$\,. The dense $^*$-subalgebra of $\Th$-smooth vectors 
$$
\Ai:=\{f\in\A\mid\Xi\ni X\mapsto\Th_X(f)\in\mathcal A\ {\rm is}\ C^\infty\}
$$ 
is also a Fr\'echet algebra with the family of semi-norms
\begin{equation}\label{semicar}
|f|_\A^k:=\sum_{|\alpha|=k}\frac{1}{\alpha!}\parallel\!\partial_X^\alpha\left[\Th_X(f)\right]_{X=0}\!\parallel_\A,\quad k\in\N\,.
\end{equation}
In \cite{Rie1,Rie2}, Marc Rieffel introduced on $\Ai$ the product
\begin{equation}\label{rodact}
f\,\#\,g:=2^{2n}\!\int_\Xi\int_\Xi e^{2i[\![Y,Z]\!]}\,\Th_Y(f)\,\Th_Z(g)dYdZ,
\end{equation}
defined as an oscillatory integral. Completing the $^*$-algebra $(\Ai,\#\,,^*)$ in a suitable $C^*$-norm $\parallel\cdot\parallel_\AA$\,, one gets a $C^*$-algebra $\AA$\,, called {\it the R-quantization of $\A$}\,. The action $\Th$ extends to a strongly continuous action on $\AA$\,, for which one keeps the same notation. The space $\AAi$ of $C^\infty$-vectors coincide with $\Ai$. 

\smallskip
The initial $C^*$-algebra $\A$ could be non-commutative, or, if it is commutative, it could consist of functions on some locally compact space $\Si$ on which $\Xi$ acts, but which is very different from it. However, the following two examples will play a special role.

\begin{example}\label{beor}
{\rm First one sets $\A:=BC_{{\rm u}}(\Xi)$\,, the $C^*$-algebra of bounded uniformly continuous functions on $\Xi$\,, which is invariant under the translations $\left[\mathcal T_X({\sf f})\right](\cdot):={\sf f}(\cdot-X)$\,. Then the $^*$-algebra of smooth vectors is $BC^\infty(\Xi)$\,,
the space of all smooth complex functions on $\Xi$ with bounded derivatives of every order. In this case, Rieffel's construction (with $\Th=\mathcal T$) is basically the standard Weyl calculus; we are going to use the special notations $\sharp$ (instead of $\#$) for the corresponding composition law and $\mathfrak B(\Xi)$ for the R-quantization of $BC_{{\rm u}}(\Xi)$.}
\end{example}

\begin{example}\label{huor}
{\rm Another option is $\A:=C_0(\Xi)$\,, the $C^*$-algebra of all the complex continuous functions on $\Xi$ that decay at infinity. Its Rieffel quantization will be denoted by $\mathfrak K (\Xi)$\,; it contains the Schwartz space $\S(\Xi)$ densely. It is known to be elementary, i.e isomorphic to the $C^*$-algebra of all compact operators in a separable Hilbert space.}
\end{example}

Following \cite{Rie1}, one introduces the Fr\'echet space $\S(\Xi;\AAi)$ composed of smooth functions $F:\Xi\rightarrow\Ai=\AAi$ with derivatives that decay rapidly with respect to all the seminorms \eqref{semicar}. We are going to use the identification of $\S\left(\Xi;\AAi\right)$ with the topological tensor product $\S(\Xi)\hat\otimes\,\AAi$ (the Fr\'echet space $\S(\Xi)$ is nuclear). On it one defines the action $\mathcal T\otimes\Th$ of the vector space $\Xi\times\Xi$ given by 
$$
\[\(\mathcal T_A\otimes\Th_Y\!\)F\]\!(X):=\Th_Y\!\[F(X-A)\],
$$ 
and the composition law (an oscillatory integral, once again in the spirit of Rieffel quantization)
\begin{equation}\label{argrur}
\(F_1\square F_2\)\!(X\!)=2^{4n}\!\!\int_\Xi\!\int_\Xi\!\int_\Xi\!\int_\Xi e^{-2i[\![A,B]\!]}\,e^{2i[\![Y,Z]\!]}
\[\left(\mathcal T_A\otimes\Th_Y\right)\!F_1\]\!(X)\[\left(\mathcal T_B\otimes\Th_Z\right)\!F_2\]\!(X)dAdBdYdZ.
\end{equation}
Supplying the involution $F^\square(\cdot):=F(\cdot)^*$, one gets a Fr\'echet $^*$-algebra.

\begin{remark}\label{uhu}
{\rm For elements $f,g\in\AAi$, ${\sf h},{\sf k}\in\S(\Xi)$ one has 
$$
({\sf h}\otimes f)\,\square\,({\sf k}\otimes g)=({\sf k}\,\sharp\,{\sf h})\otimes(f\#g)\,,
$$ 
so $\square$ is the tensor product between $\#$ and the law opposite to $\sharp$. By \cite[Prop.\,2.1]{Rie2}, one can identify $\mathfrak K(\Xi)\otimes\AA$ with the
R-deformation of $C_0(\Xi)\otimes\A\equiv C_0(\Xi;\A)$ and $\mathfrak B(\Xi)\otimes\AA$ with the R-deformation of $BC_{{\rm u}}(\Xi)\otimes\A$\,.}
\end{remark}

\begin{remark}\label{ciuhat}
{\rm We recall that $\Ai=\AAi$, but the algebraic structures are different. When the forthcoming arguments will involve the composition $\#$\,, we will use the notation $\S\left(\Xi;\AAi\right)$\,. In other situations the notation $\S\left(\Xi;\Ai\right)$ will be more natural. }
\end{remark}

Starting from the same data $\left(\A,\Th,\Xi,[\![\cdot,\cdot,]\!]\right)$\,, one can construct \cite{PR1,PR2} {\it the twisted crossed product} $C^*$-algebra $\A\rtimes_\Th^\kappa\Xi$\,. Besides the action $\Th$\,, this makes use of the group $2$-cocycle attached to the symplectic form 
\begin{equation}\label{caf}
\kappa:\Xi\times\Xi\rightarrow\mathbb T\,, \ \ \ \ \ \kappa(X,Y):=\exp\left(-\frac{i}{2}\,[\![X,Y]\!]\right),
\end{equation}
and $\A\rtimes_\Th^\kappa\Xi$ is the enveloping $C^*$-algebra of the Banach $^*$-algebra $\(L^1(\Xi;\A),\diamond,^\diamond,\parallel\cdot\parallel_1\)$\,, where
\begin{equation*}\label{aca}
\parallel\!F\!\parallel_1:=\int_\Xi \parallel F(X)\parallel_\A\!dX\,,\ \ \ \ \ F^\diamond(X):=F\(-X\)^*
\end{equation*}
and (symetrized version of the usual form)
\begin{equation}\label{ucu}
(F_1\diamond F_2)(X):=\int_\Xi\kappa(X,Y)\,\Th_{(Y-X)/2}\[G_1(Y)\]\,\Th_{Y/2}\left[F_2(X-Y)\right] dY.
\end{equation}

On $\S\left(\Xi;\AAi\right)$ we introduce {\rm the canonical mapping}
\begin{equation}\label{somorfismv}
[M(F)](X):=\int_\Xi e^{-i[\![X,Y]\!]}\,\Th_Y\!\[F(Y)\]dY,
\end{equation}
that can also be written as $M=\mathfrak F\circ{\rm C}$\,, in terms of the transformation $\[{\rm C}(F)\]\!(X):=\Th_{X}\!\[F(X)\]$ and the (symplectic) partial Fourier transform
\begin{equation*}
\mathfrak F\equiv\mathcal F\otimes 1:\S(\Xi;\AAi)\rightarrow\S(\Xi;\Ai)\,,\ \ \ \ \ (\mathfrak F F)(X):=\int_\Xi  e^{-i[\![X,Y]\!]}F(Y)dY.
\end{equation*}

We recall that $\mathfrak K(\Xi)$\,, with multiplication $\sharp$\,, has been defined in Example \ref{huor} as the R-quantization of the Abelian $C^*$-algebra $C_0(\Xi)$ on which $\Xi$ acts by translations. 
We also recall Remark \ref{uhu}. The next theorem is the main result of \cite{BM}, where several applications were indicated:

\begin{theorem}\label{stric}
\begin{enumerate}
\item
The mapping $M:\left(\S\left(\Xi;\AAi\right),\square\,,\,^\square\,\right)\rightarrow\left(\S\left(\Xi;\Ai\right), \diamond\,,\,^\diamond\,\right)$
is an isomorphism of Fr\'echet $^*$-algebras and $M^{-1}$ is its inverse.
\item
The mapping $M$ extends to a $C^*$-isomorphism $:\mathfrak K(\Xi)\otimes\AA\rightarrow \A\rtimes_\Th^\kappa\Xi$\,.
\end{enumerate}
\end{theorem}

%----------------------------------------------------------------------------------------------------------
\section{Generalized modulation spaces}\label{glamoring}
%----------------------------------------------------------------------------------------------------------

We apply now {\it localization}; this means to consider $M(F)$ for decomposable functions 
$$
F(\cdot) = ({\sf h}\otimes f)(\cdot) := {\sf h}(\cdot)f,
$$ 
with ${\sf h}\in\S(\Xi)$ and $f\in\AA^\infty$ and then to freeze ${\sf h}$ (often called {\it the window}), using this to examine $f$. For ${\sf h}\in\S(\Xi)$ we define $J_{\sf h}:\AA^\infty\to\S(\Xi;\AA^\infty)$ and $\widetilde{J}_{\sf h}:\S(\Xi;\AA^\infty)\to\AA^\infty$ by
$$
J_{\sf h}(f):={\sf h}\otimes f\,,\quad\widetilde{J}_{\sf h}(F):=\int_\Xi\overline{{\sf h}(Y)}F(Y)dY.
$$

\begin{definition}\label{cedrac}
{\rm The localized modulation map defined by} ${\sf h}\in\S(\Xi)\!\setminus\!\{0\}$ is the linear injection
\begin{equation*}\label{cedraq}
M_{\sf h}:\AA^\infty\to\S(\Xi;\A^\infty)\,,\quad M_{\sf h}(f):=(M\circ J_{\sf h})(f)=M({\sf h}\otimes f)\,.
\end{equation*}
\end{definition}

Explicitly, we get
\begin{equation*}\label{kedraq}
\big[M_{\sf h}(f)\big](X)=\int_\Xi e^{-i[\![X,Y]\!]}{\sf h}(Y)\Th_Y(f)dY.
\end{equation*}
which can also be expressed in terms of symplectic Fourier transforms and convolution: 
$$
 M_{\sf h}(f)=\mathfrak F\big[{\sf h}\Th_f\big]=\widehat{\sf h}\ast\widehat{\Th_f}\,,
 $$ 
 where one uses the notation $\Th_f:\Xi\to\A\,,\,\Th_f(X):=\Th_X(f)$\,.

\smallskip
We also set $\widetilde{M_{\sf h}}:=\widetilde{J}_{\sf h}\circ M^{-1}$. In terms of the scalar product $\<\cdot,\cdot\>_\Xi$ of $L^2(\Xi)$ (anti-linear in the first variable), one obviously has 
$$
\widetilde{M_{\sf k}}M_{\sf h}f=\widetilde{J_{\sf k}}J_{\sf h}f=\<{\sf k},{\sf h}\>_{\Xi}\,f\,,
$$
a particular case of which can be regarded as {\it an inversion formula}:
\begin{equation}\label{invform}
f=\frac{1}{\p\!{\sf h}\!\p^2_\Xi}\widetilde{M_{\sf h}}M_{\sf h}f.
\end{equation}

The localized modulation maps can be extended to $C^*$-morphisms.

\begin{corollary}\label{aragorn}
If $\,{\sf h}\,\sharp\,{\sf h}={\sf h}=\overline{\sf h}\in\S(\Xi)\!\setminus\!\{0\}$\,, then $M_{\sf h}:\big(\AA^\infty,\#,^*\big)\to\big(\S(\Xi;\A^\infty),\diamond,^\diamond\big)$ is a $^*$-monomorphism. It extends to a $C^*$-algebraic monomorphism $M_{\sf h}:\AA\to\A\rtimes_\Th^\kappa\Xi$\,.
\end{corollary}

\begin{proof}
Notice that under the stated assumptions, $J_{\sf h}$ is a $^*$-morphism:
$$
J_{\sf h}(f\#g)={\sf h}\otimes(f\#g)=({\sf h}\,\sharp\,{\sf h})\otimes(f\#g)=({\sf h}\otimes f)\,\square\,({\sf h}\otimes g)=J_{\sf h}(f)\,\square\,J_{\sf h}(g)\,,
$$
$$
J_{\sf h}(f^*)={\sf h}\otimes f^*=\overline{{\sf h}}\otimes f^*=({\sf h}\otimes f)^\square=\big(J_{\sf h}(f)\big)^\square.
$$
It is obviously injective on $\AA^\infty$. This and the first part of Theorem \ref{stric} (or rather direct computations, as in the proof of \cite[Prop.\,4.2]{BM}) easily imply the first part of the Corollary. 

\smallskip
The second part, involving the $C^*$-norms, still needs some technical effort. For convenience we indicate a direct proof, not relying on the point 2 of Theorem \ref{stric}. It is easier (but it proves less) than the proof of  \cite[Th.\,5.1]{BM}.

\smallskip
Taking into account the fact that $\S(\Xi;\A^\infty)$ is a $^*$-subalgebra of $L^1(\Xi;\A)$\,, which is in its turn a $^*$-subalgebra of the $C^*$-algebra $\A\rtimes_\Th^\kappa\Xi$\,, we examine the injective $^*$-morphism $M_{\sf h}:\AA^\infty\to\A\rtimes_\Th^\kappa\Xi$\,. We claim that it is isometric when on $\AA^\infty$ one considers the norm $\p\!\cdot\!\p_{\AA}$\,; this would insure that it can be extended to an injective $^*$-morphism on $\AA$\,. By the paragraph 3.1.6 in \cite{Bl}, this follows if $\AA^\infty$ is invariant under the $C^\infty$ functional calculus of $\AA$\,. But this property is obtained by straightforward extensions of the results of Subsection 3.2.2 in \cite{BR}, writing $\AA^\infty$ as the intersection of domains of arbitrarily large products $\delta^\alpha\!:=\delta_1^{\alpha_1}\dots\delta_{2n}^{\alpha_{2n}}$ of the closed derivations 
$$
f\to\delta_j f:=\partial_{X_j}\!\big[\Th_X(f)\big]_{X=0}\in\AA\,,\quad  j=1,...,2n:=\dim(\Xi)
$$ 
associated to the $2n$-parameter group $\Th$ of $^*$-automorphisms of $\AA$\,.
\end{proof}

We use the mappings $M$ and $M_{\sf h}$ to pull back structure. This could involve various types of topological vector spaces, but we are going to restrict our interest to normed spaces. 

\smallskip
If $\,\p\!\cdot\!\p:\S(\Xi;\A^\infty)\to\mathbb R_+$ is a norm, we define a new one by
\begin{equation*}\label{arathorn}
\p\!\cdot\!\p^M\,:\S(\Xi;\AA^\infty)\to\mathbb R_+\,,\quad\p\!F\!\p^M:=\,\p\!M(F)\!\p.
\end{equation*}
Assume now that a function ${\sf h}\in\S(\Xi)\!\setminus\!\{0\}$ (a window) is given. We define the norm
\begin{equation*}\label{galadriel}
\p\!\cdot\!\p^M_{\sf h}\,:\AA^\infty\to\mathbb R_+\,,\quad\p\!f\!\p_{\sf h}^M:=\,\p\!M_{\sf h}(f)\!\p\,=\,\p\!M({\sf h}\otimes f)\!\p\,=\,\p\!J_{\sf h}(f)\!\p^M.
\end{equation*}

\begin{definition}\label{mithrandir}
If $\mathfrak L$ denotes the completion of $\big(\S(\Xi;\A^\infty), \p\!\cdot\!\p\!\big)$\,, let us define $\mathfrak L^M$ to be the completion of $\big(\S(\Xi;\AA^\infty),\p\!\cdot\!\p^M\!\big)$  and $\mathfrak L^M_{\sf h}$ the completion of $\big(\AA^\infty,\p\!\cdot\!\p^M_{\sf h}\!\big)$\,.

\smallskip
We call $\big(\mathfrak L^M_{\sf h},\p\!\cdot\!\p^M_{\sf h}\big)$ {\rm the generalized modulation space associated to the pair} $(\mathfrak L,{\sf  h})$\,.
\end{definition}

By definition, the normed spaces $\big(\S(\Xi;\AA^\infty),\p\!\cdot\!\p^M\!\!\big)$ and $\big(\S(\Xi;\A^\infty),\p\!\cdot\!\p\!\big)$ are isomorphic, while $J_{\sf h}$ is an isometric embedding of $\big(\AA^\infty,\p\!\cdot\!\p^M_{\sf h}\!\big)$ into $\big(\S(\Xi;\AA^\infty),\p\!\cdot\!\p^M\!\big)$ and $M_{\sf h}$ an isometric embedding of $\big(\AA^\infty,\p\!\cdot\!\p^M_{\sf h}\!\big)$ into $\big(\S(\Xi;\A^\infty),\p\!\cdot\!\p\!\big)$\,. By extension one gets mappings also denoted by $M:\mathfrak L^M\to\mathfrak L$ (an isomorphism) and $M_{\sf h} : \mathfrak L^M_{\sf h}\to\mathfrak L$ (an isometric embedding). Often a Banach space $(\mathfrak L,\p\!\cdot\!\p)$ containing densely $\S(\Xi;\A^\infty)$ is given and one applies the procedure above to induce a Banach space $\mathfrak L^M_{\sf h}$ containing $\AA^\infty$ densely. The denseness of $\S(\Xi;\A^\infty)$ could be avoided using extra techniques, but this will not be done here.

\smallskip
Concerning the compatibility of norms with $^*$-algebra structures we can say basically that, using a self- adjoint idempotent window, one induces Banach $^*$-algebras from Banach $^*$-algebras and $C^*$-algebras from $C^*$-algebras:

\begin{proposition}\label{gandalf}
Assume that ${\sf h}\ne 0$ is a self-adjoint projection in $(\S(\Xi),\sharp)$\,, i.e. ${\sf h}\,\sharp\,{\sf h}={\sf h}=\overline{\sf h}$\,.

\begin{enumerate}
\item
If the involution $^\diamond$ in $\big(\S(\Xi;\A^\infty),\p\!\cdot\!\p\!\big)$ is isometric, the involution $^*$ in $\big(\AA^\infty,\p\!\cdot\!\p^M_{\sf h}\!\big)$ is also isometric and it extends to an isometric involution on $\mathfrak L^M_{\sf h}$.
\item
If $\,\p\!\cdot\!\p$ is sub-multiplicative with respect to $\diamond$\,, then $\p\!\cdot\!\p^M_{\sf h}$ is sub-multiplicative with respect to $\#$\,. The completion $\mathfrak L^M_{\sf h}$ becomes a Banach algebra sent isometrically by $M_{\sf h}$ into the Banach algebra $\mathfrak L$\,.
\item
If $\,\p\!\cdot\!\p$ is a $C^*$-norm, then $\p\!\cdot\!\p^M_{\sf h}$ is also a $C^*$-norm and $\mathfrak L^M_{\sf h}$ is a $C^*$-algebra, which can be identified with a $C^*$-subalgebra of $\mathfrak L$\,.
\end{enumerate}
\end{proposition}

\begin{proof}
This follows easily from the fact that $M_{\sf h}$ is a $^*$-monomorphism, cf. Corollary \ref{aragorn}. Let us check the second item, for instance:
$$
\p\!f\#g\!\p^M_{\sf h}=\,\p\!M_{\sf h}(f\#g)\!\p\,=\,\p\!M_{\sf h}(f)\diamond M_{\sf h}(g)\!\p\,\le\,\p\!M_{\sf h}(f)\!\p\p\!M_{\sf h}(g)\!\p\,=\,\p\!f\!\p^M_{\sf h}\p\!g\!\p^M_{\sf h}.
$$
\end{proof}

\begin{example}\label{sauron}
{\rm For any $p\in[1,\infty)$ one can consider the Banach space $\mathfrak L:=L^p(\Xi;\A)$ with norm
$$
\p\!F\!\p\,:=\Big(\int_\Xi \p\!F(X)\!\p^p_\A dX\Big)^{1/p}\,.
$$
Among the generalized modulation spaces $\big[L^p(\Xi;\A)\big]^M_{\sf h}$, those with $p=1$ and ${\sf h}\,\sharp\,{\sf h}={\sf h}=\overline{\sf h}$ are Banach $^*$-algebras.
}
\end{example}

We treat now the problem of the ${\sf h}$-dependence of the Banach space $\mathfrak L^M_{\sf h}$. We say that the norm $\p\!\cdot\!\p$ on $\S(\Xi;\A^\infty)$ is {\it admissible} (and call the completion $\mathfrak L$ {\it an admissible Banach space}) if for any ${\sf h},{\sf k}\in\S(\Xi)\!\setminus\!\{0\}$ the operator 
$$
R_{\sf k,h}:=M_{\sf k}\widetilde M_{\sf h}:\big(\S(\Xi;\A^\infty),\p\!\cdot\!\p\!\big)\to\big(\S(\Xi;\A^\infty),\p\!\cdot\!\p\!\big)
$$ 
is bounded. 

\begin{proposition}\label{fagorn}
If for a fixed couple $({\sf h},{\sf k})$ the operator $R_{\sf k,h}$ is bounded, we get a continuous dense embedding $\mathfrak L^M_{\sf h}\to\mathfrak L^M_{\sf k}$. So, if $\mathfrak L$ is admissible, all the Banach spaces $\big\{\mathfrak L^M_{\sf h}\mid{\sf h}\in\S(\Xi)\!\setminus\!\{0\}\big\}$ are isomorphic.
\end{proposition}

\begin{proof}
It is enough to show that for some positive constant $C({\sf h},{\sf k})$ one has $\p\!f\!\p^M_{\sf k}\,\le C({\sf h},{\sf k})\!\p\!f\!\p^M_{\sf h}$ for all $f\in\S (\Xi; \A^\infty)$\,. This follows from the assumption and from \eqref{invform}:
$$
\p\!f\!\p^M_{\sf k}\,=\,\p\!M_{\sf k}f\!\p\,=\,\frac{1}{\p\!{\sf h}\!\p_\Xi^2}\p\!M_{\sf k}\big(\widetilde M_{\sf h}M_{\sf h}f\big)\!\p\,\le\frac{\p\!M_{\sf k}\widetilde M_{\sf h}\!\p}{\p\!{\sf h}\!\p_\Xi^2}\p\!M_{\sf h}f\!\p\,=\frac{\p\!R_{\sf k,h}\!\p}{\p\!{\sf h}\!\p_\Xi^2}\p\!f\!\p^M_{\sf h}.
$$
\end{proof}

One deduces from Corollary \ref{aragorn} that $\big[\A\rtimes_\Th^\kappa\Xi\big]^M_{\sf h}\!=\AA\,$ for any idempotent window ${\sf h}$\,; in this case the norm is really ${\sf h}$-independent.

\begin{remark}\label{saruman}
{\rm Recall the expression $R_{\sf k,h}:=M_{\sf k}\widetilde M_{\sf h}=M J_{\sf k}\widetilde J_{\sf h}M^{-1}$. Since $M$ and $M^{-1}$ are isomorphisms, the real issue is whether 
$$
J_{\sf k}\widetilde J_{\sf h}:\big(\S(\Xi;\AA^\infty),\p\!\cdot\!\p^M\!\big)\to\big(\S(\Xi;\AA^\infty),\p\!\cdot\!\p^M\!\big)
$$
is bounded or not. If one has a good understanding of the norm $\p\!\cdot\!\p^M$, the verification becomes easier, since  $J_{\sf k}\widetilde J_{\sf h}=I_{{\sf k},{\sf h}}\otimes{\rm id}$\,, where $I_{{\sf k},{\sf h}}$ is just the integral operator with kernel ${\sf k}\otimes\overline{\sf h}$ (a rank one operator).
}
\end{remark}

%----------------------------------------------------------------------------------------------------------
\section{Representations}\label{glamorama}
%----------------------------------------------------------------------------------------------------------

We turn now to representations, always supposed to be non-degenerate. The natural Hilbert space realization of a twisted $C^*$-dynamical system $(\A,\Th,\Xi,\kappa)$ is achieved by covariant representations $(r,T,\H)$\,, where $r$ is a representation of $\A$ in the $C^*$-algebra $\mathbb B(\H)$ of all bounded linear operators in $\H$\,, $ T$ is a strongly continuous unitary projective representation in $\H$\,:
\begin{equation}\label{medhros}
T(X)T(Y)=\kappa(Y,X)T(X+Y)\,,\quad\forall\,X,Y\in\Xi\,,
\end{equation}
and for any $Y\in\Xi$ and $g\in\A$ one has
\begin{equation}\label{arwen}
T(Y)r(g)T(-Y)=r\big[\Th_Y(g)\big]\,.
\end{equation}
Hilbert-space representations of the twisted crossed product  (the most general, actually) $\,r\rtimes T :\A\rtimes_\Th^\kappa \Xi\to\mathbb B(\H)$ are associated to covariant representations $(r, T ,\H)$ of $(\A, \Th, \Xi, \kappa)$ by
\begin{equation}\label{gangee}
(r\rtimes T)(G):=\int_\Xi r\big\{\Th_{X/2}[G(X)]\big\}T(X)dX,\quad G\in L^1(\Xi;\A)\,.
\end{equation}
The localized modulation mappings allow us to use in a particular way covariant representations of the initial data in the representation theory of the $R$-quantized $C^*$-algebra $\AA$\,. 

\smallskip
Let $(r, T,\H)$ be a covariant representations for $(\A,\Th,\Xi,\kappa)$ and ${\sf h}\sharp{\sf h} = {\sf h} = {\sf h}\in\S(\Xi)\!\setminus\!\{0\}$ any idempotent window. Composing $r\rtimes T :\A\rtimes_\Th^\kappa \Xi\to\mathbb B(\H)$ with the $^*$-morphism $M_{\sf h} : \AA\to\A\rtimes_\Th^\kappa\Xi$ (cf. Corollary 2.4) one gets the representation
\begin{equation*}\label{frodo}
(r\rtimes T)^M_{\sf h}:=(r\rtimes T)\circ M_{\sf h}:\AA\to\mathbb B(\H)\,,
\end{equation*}
that is given on $\AA^\infty$ by
\begin{equation*}\label{meriadoc}
(r\rtimes T)^M_{\sf h}\!(f)=\int_\Xi\int_\Xi e^{-i[\![X,Y]\!]}{\sf h}(Y)r\big\{\Th_{Y+X/2}(f)\big\}T(X)dXdY.
\end{equation*}
If $r\rtimes T$ is faithful, $(r\rtimes T )^M_{\sf h}$ is faithful too, since $M_{\sf h}$ is injective. Unitary equivalence is preserved under the correspondence $(r, T ) \to (r\rtimes T )^M_{\sf h}$.

\begin{proposition}\label{pippin}
Along the $\mathcal T$-orbits, the representations $(r\rtimes T)^M_{\sf h}$ are unitarily equivalent.
\end{proposition}

\begin{proof}
We recall the notation $(\mathcal T_Z{\sf h})(\cdot) = {\sf h}(\cdot-Z)$ and notice that ${\sf h}$ and $\mathcal T_Z{\sf h}$ are simultaneously self-adjoint projections. 
By \eqref{caf} and \eqref{medhros}, one checks immediately that 
$$
T(X)T(-Z)=e^{-i[\![X,Z]\!]}\,T(-Z)T(X)\,,\quad\forall\,X,Z\in\Xi\,.
$$
Using this identity,  equation \eqref{arwen} and a change of variables, one computes for $f\in\AAi$
$$
\begin{aligned}
T(Z)(r\rtimes T)^M_{\sf h}\!(f)T(-Z)&=\int_\Xi\!\int_\Xi e^{-i[\![X,Y]\!]}{\sf h}(Y)T(Z)r\big\{\Th_{Y+X/2}(f)\big\}T(X)T(-Z)dXdY\\
&=\int_\Xi\!\int_\Xi e^{-i[\![X,Y]\!]}{\sf h}(Y)T(Z)r\big\{\Th_{Y+X/2}(f)\big\}e^{-i[\![X,Z]\!]}T(-Z)T(X)dXdY\\
&=\int_\Xi\!\int_\Xi e^{-i[\![X,Y+Z]\!]}{\sf h}(Y)r\big\{\Th_{Y+Z+X/2}(f)\big\}T(X)dXdY\\
&=\int_\Xi\!\int_\Xi e^{-i[\![X,Y]\!]}{\sf h}(Y-Z)r\big\{\Th_{Y+X/2}(f)\big\}T(X)dXdY.
\end{aligned}
$$
This can be written
\begin{equation*}\label{radagast}
T(Z)(r\rtimes T)^M_{\sf h}\!(f)T(-Z)=(r\rtimes T)^M_{\mathcal T_Z{\sf h}}(f)\,,\quad\forall\,Z\in\Xi\,,\ f\in\AA^\infty,
\end{equation*}
and by density this also holds for $f\in\AA$\,. Thus $(r\rtimes T)^M_{\sf h}$ and $(r\rtimes T)^M_{\mathcal T_Z{\sf h}}(f)$ are unitarily equivalent.
\end{proof}

Actually, starting with an arbitrary representation $\rho:\A\to\mathbb B(\K)$\,, one can induce canonically a covariant representation $\big(r\!_\rho, T,\H)$ of $(\A,\th,\Xi,\kappa)$\,, setting $\H := L^2(\Xi;\mathcal K)$\,,
$$
\big[r\!_\rho(f)\Phi\big](X):=\rho\big[\Th_X(f)\big][\Phi(X)]\,,\quad f\in\A\,,\ X\in\Xi\,,\ \Phi\in L^2(\Xi;\mathcal K)
$$
and
$$
[T(Y)\Phi](X):=\kappa(Y,X)\Phi(X+Y)\,,\quad X,Y\in\Xi\,,\ \Phi\in L^2(\Xi;\mathcal K)\,.
$$
Then one associates the representations $\rho_{(M,{\sf h})} := (r\!_\rho\rtimes T)^M_{\sf h}$ of $\AA$ in $L^2(\Xi;\K)$ indexed by the non-null self- adjoint projections in $(\S(\Xi),\sharp)$\,. It is straightforward to check that the correspondence $\rho\to\rho_{(M,{\sf h})}$ preserves unitary equivalence.

\smallskip
The $C^*$-norm on $\AA$ has been defined in \cite{Rie1} by Hilbert module techniques. The next result supplies alternative formulae

\begin{proposition}\label{legolas}
\begin{enumerate}
\item
For any idempotent real window ${\sf h}\,\sharp\,{\sf h} = {\sf h}\in\S(\Xi)\!\setminus\!\{0\}$ and for each $f \in\AA$ one has
\begin{equation*}
\p\!f\!\p_\AA=\sup\big\{\!\p\!(r\rtimes T)[M_{\sf h}(f)]\!\p_{\mathbb B(\H)}\,\mid (r,T,\H)\ {\rm covariant\ representation\ of}\ (\A,\th,\Xi,\kappa)\big\}\,.
\end{equation*}
\item
Moreover, for any faithful representation $\rho : \A\to\mathbb B(\K)$ one has
\begin{equation*}\label{gimli}
\p\!f\!\p_\AA\,=\,\p\!\rho_{(M,{\sf h})}(f)\!\p_{\mathbb B[L^2(\Xi;\K)]}.
\end{equation*}
\end{enumerate}
\end{proposition}

\begin{proof}
The two formulas follow from the fact that $M_{\sf h} : \AA\to\A\rtimes^\kappa_\Th \Xi$ is an isometry and from the well-known forms of the universal and the reduced norm in twisted crossed products \cite{PR1,PR2}, that coincide since the group $\Xi$ is Abelian, thus amenable.
\end{proof}

%----------------------------------------------------------------------------------------------------------
\section{The Abelian case}\label{glamodrama}
%----------------------------------------------------------------------------------------------------------

If $\A$ is Abelian, by Gelfand theory, it is isomorphic (and will be identified) to $C_0(\Si)$\,, the $C^*$-algebra of all complex continuous functions on the locally compact space $\Si$ that converge to zero at infinity. The space $\Si$ is a homeomorphic copy of the Gelfand spectrum of $\A$ and it is compact iff $\A$ is unital. Then the group $\Th$ of automorphisms is induced by an action (also called $\Th$) of $\Xi$ by homeomorphisms of $\Si$\,. We use the convention
$$
[\Th_X(f)](\si) := f [\Th_X(\si)]\,,\quad \forall\,\si\in\Si\,,\, X\in\Xi\,,\, f\in\A\,,
$$
as well as the notation $\Th_X(\si) = \Th(X,\si) = \Th_\si(X)$ for the $X$-transform of the point $\si$.

\smallskip
Assuming that $\A\equiv C_0(\Si)$ is Abelian, we set $\AA =: \CC_0(\Si)$ for the (non-commutative) Rieffel $C^*$-algebra associated to $C_0(\Si)$ by quantization and $\CC_0^\infty(\Si) = C_0^\infty(\Si)$ for the (common) space of smooth vectors under the action $\Th$\,. Of course, this is just a matter of notation: $\Si$ do not possess an intrinsic smooth structure and $\CC_0(\Si)$ is most often non-commutative.

\smallskip
For each $\si\in\Si$\,, we introduce a concrete covariant representation $\(r_\rho,T,L^2(\Xi)\)$ of the twisted dynamical system $\big(C_0(\Si), \Th, \Xi, \kappa\big)$ by
\begin{equation}\label{thorin}
r_\si:C_0(\Si)\to\mathbb B\big[L^2(\Xi)\big]\,,\quad\big[r_\si(g)\big](X):=g\big[\Th_X(\si)\big]\Phi(X)
\end{equation}
and
\begin{equation}\label{bilbo}
T(Y):L^2(\Xi)\to L^2(\Xi)\,,\quad [T(Y)\Phi](X):=\kappa(Y,X)\Phi(X+Y)\,.
\end{equation}
It is induced from the one dimensional representation 
$$
\rho_\si : C_0(\Si)\to\mathbb B(\mathbb C)\cong\mathbb C\,,\quad\rho_\si(f) := f(\si)\,.
$$ 
The general procedure of the Section \ref{glamorama} provides a family of representations $(r_\si\!\rtimes T )^M_{\sf h}$ of $\AA$ in the Hilbert space $L^2(\Xi)$\,, indexed by the non-null projections of $(\S(\Xi),\sharp)$\,.

\smallskip
To connect these representations with rather familiar Weyl-type operators, we need first to recall somehow informally some basic facts about {\it the standard Weyl quantization} $\,{\sf f}\to\mathfrak{Op}({\sf f})$ \cite{Fo}. We assume that $\Xi =\mathscr X\times\mathscr X^*$, with points $X=(x,\xi),Y=(y,\eta),\dots$ The action of $\mathfrak{Op}({\sf f})$ on $\S(\mathscr X)$ or $\H := L^2(\mathscr X)$ (under various assumptions and with various interpretations) is given by
\begin{equation*}\label{gollum}
[\mathfrak{Op}({\sf f})v](x):=\int_\mathscr X\!\int_{\mathscr X^*}\!\!e^{i(x-y)\cdot\xi}\,{\sf f}\Big(\frac{x+y}{2},\xi\Big)v(y)dyd\xi\,.
\end{equation*}
We recall that $\mathfrak{Op}(f\sharp\,g) = \mathfrak{Op}(f)\mathfrak{Op}(g)$ and $\mathfrak{Op}\big(\overline{f}\big) = \mathfrak{Op}(f)^*$. It is useful to introduce the family of unitary operators 
\begin{equation}\label{bombadill}
\mathfrak{op}(X)=\mathfrak{Op}({\sf e}_X)\,,\quad {\sf e}_X(Y):=e^{-i[\![X,Y]\!]}\,,
\end{equation}
satisfying 
$$
\mathfrak{op}(X)\mathfrak{op}(Y)=\kappa(X,Y)\mathfrak{op}(X+Y)\,,\quad\forall\,X,Y\in\Xi\,.
$$

Using these, one gets a family $\{\mathfrak{Op}_\si \mid \si\in\Si\}$ of Schr\"odinger-type representations of the Rieffel $C^*$-algebra $\CC_0(\Si)$ in the Hilbert space $\H = L^2(\mathscr X)$\,, indexed by the points of the space of the dynamical system. They are given for $f\in\CC_0^\infty(\Si)$ by $\mathfrak{Op}_\si(f ) := \mathfrak{Op}[f\circ\Th_\si]$\,; using oscillatory integrals one may write
\begin{equation*}\label{gollumus}
[\mathfrak{Op}(f)u](x):=\int_\mathscr X\!\int_{\mathscr X^*}\!\!e^{i(x-y)\cdot\xi}\,f\Big[\Th_{\big(\frac{x+y}{2},\xi\big)}(\si)\Big]u(y)dyd\xi\,,\quad u\in L^2(\mathscr X)\,.
\end{equation*}

The extension from $\CC_0^\infty(\Si)$ to $\CC_0(\Si)$ is slightly non-trivial, but it is explained in \cite{Ma}. Note that if $\si$ and $\si'$ belong to the same $\Th$-orbit, the representations $\mathfrak{Op}_\si$ and $\mathfrak{Op}_{\si'}$ are unitarily equivalent. $\mathfrak{Op}_\si$ is faithful if and only if the orbit generated by $\si$ is dense. The justifications and extra details can be found in \cite{Ma}.

\begin{remark}\label{elf}
{\rm It is easy to see that $\big(\mathfrak{Op}_\si,\mathfrak{op},L^2(\mathscr X)\big)$ is a covariant representation of $\big(\CC_0(\Si),\Th,\Xi,\kappa\big)$\,. This follows applying $\mathfrak{Op}$ to the relations
$$
{\sf e}_Y\sharp\,{\sf e}_Z=\kappa(Y,Z){\sf e}_{X+Y}\,,\quad{\sf e}_Y\sharp(f\circ\Th_\si)\sharp\,{\sf e}_{-Y}=[\Th_Y(f)]\circ\Th_\si\,.
$$
}
\end{remark}

We would like now to make the connection between the representations $\mathfrak{Op}_\si$ and $(r_\si\rtimes T)^M_{\sf h}$ of the Rieffel algebra $\CC_0(\Si)$ for convenient idempotent windows. This needs some preparations involving the Bargmann transform.

\smallskip
For various types of vectors $u, v : \mathscr X\to\mathbb C$ we define {\it the Wigner transform} ($V$) and {\it the Fourier-Wigner transform} ($W$) by 
$$
W_{u,v}(X) = \<u,\mathfrak{op}(X)v\>_\mathscr X\quad {\rm and}\quad V_{u,v} = \mathcal FW_{u,v}\,.
$$ 
Their important role is shown by the relations
\begin{equation}\label{elrond}
\<u,\mathfrak{Op}({\sf f})v\>_\mathscr X=\int_\Xi{\sf f}(X)V_{u,v}(X)dX,\quad\<u,\mathfrak{Op}({\sf f})v\>_\mathscr X=\int_\Xi(\mathcal F{\sf f})(X)W_{u,v}(X)dX.
\end{equation}

Let us fix $v\in\S(\mathscr X)$ with $\p\!v\!\p_{\mathscr X}\,=\!1$\,. For any $Y\in\Xi$ we define 
$$
v(Y):=\mathfrak{op}(-Y)v\in\H
$$ 
({\it the family of coherent vectors associated to $v$})\,. The isometric mapping $\mathcal U_v : L^2(\mathscr X)\to L^2(\Xi)$ given by
\begin{equation*}\label{dorlomin}
(\mathcal U_v u)(X):=\<v(X),u\>_\mathscr X=\<v,\mathfrak{op}(X)u\>\!_{\mathscr X}=W_{u,v}(X)
\end{equation*}
is called {\it the (generalized) Bargmann transformation corresponding to the family of coherent states} $\{v(X) |X\in\Xi\}$\,. Its adjoint is given by
\begin{equation*}\label{ungoliant}
\mathcal U_v^*\Phi=\int_\Xi\Phi(Y)v(Y)dY\,,\quad\forall\,\Phi\in L^2(\Xi)\,.
\end{equation*}
We also set $\mathbb U_v[T] := \mathcal U_v T\mathcal U_v^*$ for any $T\in\mathbb B[L^2(\mathscr X)]$\,. Now take 
$$
{\sf h}\equiv{\sf h}(v) := V_{v,v}\in\S(\Xi)
$$ 
with explicit form
$$
[{\sf h}(v)](x,\xi)=\int_\mathscr X \!e^{iy\cdot\xi}\,\overline{v\Big(x+\frac{y}{2}\Big)}v\Big(x-\frac{y}{2}\Big)dy\,.
$$
Then ${\sf h}(v)\,\sharp\,{\sf h}(v)=\overline{{\sf h}(v)}={\sf h}(v)$ and $\mathfrak{Op}({\sf h}(v))$ will be the rank-one projection $|v\>\<v|$\,.

\begin{theorem}\label{dwarf}
For any $\si\in\Si$ and any normed vector $v$\,, one has on the $R$-quantization $\CC_0(\Si)\equiv\AA$ of $C_0(\Si)$
\begin{equation*}\label{hobbit}
(r_\si\!\rtimes T)^M_{{\sf h}(v)}=\mathbb U_v\circ\mathfrak{Op}_\si\,.
\end{equation*}
\end{theorem}

\begin{proof}
By density, it is enough to compute on $\CC_0^\infty(\Si)$\,. Using successively the expressions of $\mathcal U_v$\,,\,$\mathfrak{Op}_\si$\,,\,$\mathcal U_v^*$, formulas \eqref{bombadill} and \eqref{elrond} and the fact that ${\sf h}(v)$ is real, we get
$$
\begin{aligned}
\big[\mathcal U_v\mathfrak{Op}_\si(f)\mathcal U_v^*\Phi\big](X)&=\<v(X),\mathfrak{Op}_\si(f)\,\mathcal U^*_v\Phi\>\!_{\mathscr X}\\
&=\Big\<v(X),\mathfrak{Op}(f\circ\Th_\si)\!\int_\Xi\Phi(Y)v(Y)dY\Big\>\!_{\mathscr X}\\
&=\int_\Xi\Phi(Y)\<v(X),\mathfrak{Op}(f\circ\Th_\si)v(Y)\>\!_{\mathscr X}dY\\
&=\int_\Xi\Phi(Y)\<v,\mathfrak{Op}\big({\sf e}_{X}\sharp[f\circ\Th_\si]\sharp{\sf e}_{-Y}\big)v\>\!_{\mathscr X}dY\\
&=\int_\Xi\<{\sf h}(v),{\sf e}_{X}\sharp[f\circ\Th_\si]\sharp{\sf e}_{-Y}\>_{\Xi}\,\Phi(Y)dY.
\end{aligned}
$$
On the other hand, by \eqref{gangee}, \eqref{thorin}, \eqref{bilbo}, a change of variables and the explicit form of $M_{{\sf h}(v)}$
$$
\begin{aligned}
\Big[(r_\si\!\rtimes T)^M_{{\sf h}(v)}(f)\Phi\Big](X)&=\int_\Xi\kappa(Z,X)\big[M_{{\sf h}(v)}(f)\big]\big(\Th_{X+Z/2}(\si),Z\big)\Phi(X+Z)dZ\\
&=\int_\Xi\kappa(Y,X)\big[M_{{\sf h}(v)}(f)\big]\big(\Th_{(X+Y)/2}(\si),Y-X\big)\Phi(Y)dY\\
&=\int_\Xi\kappa(Y,X)\!\int_\Xi e^{-i[\![Y-X,Z]\!]}[{\sf h}(v)](Z)f\big[\Th_{Z+(X+Y)/2}(\si)\big]\Phi(Y)dZdY.
\end{aligned}
$$
Thus it is enough to show that for all $f\in C_0^\infty(\Si)\,,\,{\sf h}=\overline{\sf h}\in\S(\Xi)\,,\,X,Y\in\Xi\,,\,\si\in\Si$ one has
$$
\<{\sf h},{\sf e}_{X}\sharp(f\!\circ\!\Th_\si)\sharp{\sf e}_{-Y}\>_{\Xi}=\kappa(Y,X)\!\int_\Xi e^{-i[\![Y-X,Z]\!]}{\sf h}(Z)f\big[\Th_{Z+(X+Y)/2}(\si)\big]dZ\,.
$$
This amounts to
$$
\big({\sf e}_{X}\sharp(f\!\circ\!\Th_\si)\sharp{\sf e}_{-Y}\big)(Z)=\kappa(Y,X)e^{-i[\![Y-X,Z]\!]}f\big[\Th_\si(Z+(X+Y)/2)\big]\,,\quad\forall\,Z\in\Xi\,,
$$
which follows from a straightforward computation of the left-hand side.
\end{proof}

We discuss shortly "orthogonality matters". On $\Si$ we pick a $\Th$-invariant measure $d\si$ and work with scalar products of the form
$$
\<f,g\>_\Si:=\int_\Si \overline{f(\si)}g(\si) d\si\,,\quad \<F,G\>_{\Xi\times\Si}:=\int_\Xi\int_\Si \overline{F(X,\si)}G(X,\si) dXd\si\,.
$$ 
The relationship between the spaces $\S\big(\Xi;C_0^\infty(\Si)\big)$ and $L^2(\Xi\times\Si)$ depends on the assumptions we impose on $(\Si,d\si)$\,. If $d\si$ is a finite measure, for instance, one has $\S\big(\Xi;C_0^\infty(\Si)\big)\subset L^2(\Xi\times\Si)$\,. Anyhow, the canonical map can be defined independently on $L^2(\Xi\times\Si)$\,.

\begin{proposition}\label{one}
One has {\rm the orthogonality relations} valid for $F,G\in L^2(\Xi\times\Si)$\,:
\begin{equation*}\label{tion}
\<M(F),M(G)\>_{\Xi\times\Si}=\<F,G\>_{\Xi\times\Si}\,.
\end{equation*}
Thus the operator $M:L^2(\Xi\times\Si)\rightarrow L^2(\Xi\times\Si)$ is unitary.
\end{proposition}

\begin{proof}
It is enough to recall the definition $M=\mathfrak F\circ{\rm C}$\,. The (symplectic) partial Fourier transformation $\mathfrak F$ is unitary in $L^2(\Xi\times\Si)\cong L^2(\Xi)\otimes L^2(\Si)$\,. The invertible mapping ${\rm C}$ reads in this case
$$
[{\rm C}(F)](X,\si)=F\big(X,\Th_X(\si)\big)
$$
and is also an isomorphism of $L^2(\Xi\times\Si)$\,, since $d\si$ is $\Th$-invariant.
\end{proof}

In this setting, one can use mixed Lebesgue spaces $L^{p,q}(\Xi\times\Si)$ to induce modulation spaces.

\bigskip
\medskip
{\bf Acknowledgements:}
{M. M\u antoiu has been  supported by {\it Proyecto Fondecyt No.\,1160359} and by {\it N\'ucleo Milenio de F\'isica Matem\'atica RC120002.}
}


\begin{thebibliography}{99}

\bibitem{BB} I.~Belti\c t\u a and D.~Belti\c t\u a: Modulation Spaces of Symbols for Representations of Nilpotent Lie Groups, \emph{J. Fourier Analysis Appl.} \textbf{17} (2), 290--319, (2011).

\bibitem{BB4} I.~Belti\c t\u a and D.~Belti\c t\u a:  Continuity of Magnetic Weyl Calculus, {\em J. Funct. Analysis}, \textbf{260}, 1944--1968, (2011).

\bibitem{BM} I. Bellti\c t\u a and M. M\u antoiu: Rieffel Deformation and Twisted Crossed Products, \emph{Int. Math. Res. Notices}, 551--567, (2014).

\bibitem{Bl} B.~Blackadar, {\em K-Theory for Operator Algebras}, Springer-Verlag, New York/ Berlin/ Heidelberg, (1986).

\bibitem{Bu} A. Boulkhemair: Remarks on a Wiener Type Pseudodifferential Algebra and Fourier Integral Operators, \emph{Math. Res. Lett.} \textbf{4} (1), 53--67, (1997).

\bibitem{BR} O. Brattelli and D. W. Robinson: \emph{Operator Algebras and Quantum Statistical Mechanics I}, Texts and Monographs in Physics, Springer, (2002).

\bibitem{Fe} H. G. Feichtinger: On a New Segal Algebra, \emph{Monatsh. Mat.} \textbf{92} (4), 269--289, (1981).

\bibitem{Fe1} H. G. Feichtinger: Modulation Spaces on Locally Compact Abelian Groups, In Proceedings of "International Conference on Wavelets and Applications" 2002, pages 99--140, Chenai, India. Updated version of a technical report, University of Viena, (1983).

\bibitem{FG} H. G. Feichtinger and K. Gr\"ochenig: Banach Spaces Associated to Integrable Group Representations and Their Atomic Decompositions I, \emph{J. Funct. Analysis}, \textbf{86}, 307--340, (1989). 

\bibitem{Fo} G. B. Folland: \emph{Harmonic Analysis in Phase Space}, Princeton Univ. Press, Princeton, NJ, (1989).

\bibitem{FR} M. Fornasier and H. Rauhut: Continuous Frames, Fuction Spaces and the Discretization Problem, \emph{J. Fourier Anal. Appl.} \textbf{11} (3), 245--287, (2005). 

\bibitem{Gr} K. Gr\"ochenig: \emph{Foundations of Time-Frequency Analysis}, Birkh\"auser Boston Inc., Boston, MA, (2001).

\bibitem{Gr1} K. Gr\"ochenig: Time-Frequency Analysis of Sj\"ostrand Class, \emph{Revista Mat. Iberoam.} \textbf{22} (2), 703--724, (2006).

\bibitem{Gr2} K. Gr\"ochenig: A Pedestrian Approach to Pseudodifferential Operators in \emph{Harmonic Analysis and Applications}, C. Heil editor, Birkh\"auser, Boston, (2006). In Honour of John J. Benedetto.

\bibitem{GH} K. Gr\"ochenig and C. Heil: Modulation Spaces and Pseudodifferential Operators, \emph{Int. Eq. Op. Th.} \textbf{34}, 439--457, (1999).

\bibitem{GR} K. Gr\"ochenig and Z. Rzeszotnik: Banach Algebras of Pseudodifferential Operators and Their Almost Diagonalization, \emph{Ann. Inst. Fourier}, \textbf{58} (6), 2279--2314, (2008).

\bibitem{GS} K. Gr\"ochenig and T. Sthromer: Pseudodifferential Operators on Locally Compact Abelian Groups and Sj\"ostrand's Symbol Class, \emph{J. Reine Angew. Math.} \textbf{613}, 121--146, (2007).

\bibitem{HRT} C. Heil, J. Ramanathan and P. Topiwala: Singular Values of Compact Pseudodifferential Operators, \emph{J. Funct. Anal.}, \textbf{150} (2), 426--452, (1997).

\bibitem{Kp} P. Kasprzak: Rieffel Deformation via Crossed Products, {\em J. Funct. Anal.} {\bf 257} (1), 1288--1332, (2009).

\bibitem{LdG} F. Luef and M. de Gosson: On the Usefulness of Modulation Spaces in Deformation Quantization, \emph{J. Phys. A: Math. Theor.} \textbf{42}, 315205--315221, (2009).

\bibitem{Ma} M. M\u antoiu: Rieffel's Pseudodifferential Calculus and Spectral Analysis for Quantum Hamiltonians, {\em Ann. Inst. Fourier}, \textbf{62}, 1551--1580, (2012).

\bibitem{Ma1} M. M\u antoiu: On the Essential Spectrum of Phase-Space Anisotropic Pseudodifferential Operators, \emph{Math. Proc. Cam. Phil. Soc.} \textbf{154}, 29--39, (2013).

\bibitem{Ma2} M. M\u antoiu: Coorbit Spaces of Symbols for Square Integrable Families of Operators, \emph{Math. Rep.} \textbf{18}, 63--83, (2016).

\bibitem{MP} M. M\u antoiu and R. Purice: The Modulation Mapping  for Magnetic Symbols and Operators, \emph{Proc. Amer. Math. Soc} \textbf{138} (8), 2839--2852, (2009).

\bibitem{MP1} M. M\u antoiu and R. Purice: Abstract Composition Laws and Their Modulation Spaces, \emph{J. Pseudodiff. Op. and Appl.} \textbf{3} (3), 283--307, (2012).

\bibitem{Ne} S. Neshveyev: Smooth Crossed Products of Rieffel's Deformations, \emph{Letters in Math. Phys.} \textbf{104} (3), 361--371, (2014).

\bibitem{PR1} J. Packer and I. Raeburn: Twisted Crossed Products of $C^*$-algebras,  {\em Math. Proc. Camb. Phyl. Soc.} {\bf 106}, 293--311, (1989).

\bibitem{PR2} J. Packer and I. Raeburn: Twisted Crossed Products of $C^*$-algebras, II, {\em Math. Ann.} {\bf 287}, 595--612, (1990).

\bibitem{RU} H. Rauhut and T. Ullrich: Generalized Coorbit Space Theory and Inhomogeneous Function Spaces of Besov-Lizorkin-Triebel Type, \emph{J. Funct. Analysis}, \textbf{260} (11), 3299--3362, (2011). 

\bibitem{Rie1} M. A. Rieffel: Deformation Quantization for Actions of $\,\mathbb R^d$, {\em Memoirs of the AMS}, {\bf 506}, (1993).

\bibitem{Rie2} M. A. Rieffel: Compact Quantum Groups Associated with Toral Subgroups. {\em Representation Theory of Groups and Algebras.} Contemp. Math. {\bf 145},
Providence RI: Am. Math. Soc. 465--491, (1993).

\bibitem{Sj} J. Sj\"ostrand: An Algebra of Pseudodifferential Operators, \emph{Math. Res. Lett.} \textbf{1} (2), 185--192, (1994).

\bibitem{Sj1}  J. Sj\"ostrand: Wiener Type Algebras of Pseudodifferential Operators, In \emph{S\'eminaire sur les \'Ecuations aux D\'eriv\'ees Partielles}, 1994--1995, pages Exp. No. IV, 21, \'ecole Polytech., Palaiseau, (1995).

\bibitem{T}  J. Toft: Subalgebras to a Wiener Type Algebra of Pseudodifferential Operators, \emph{Ann. Inst. Fourier}, \textbf{51} (5), 1347--1383, (2001).

\bibitem{To} J. Toft: Continuity Properties for Modulation Spaces, with Applications to Pseudodifferential Calculus I, \emph{J. Funct. Analysis}, \textbf{207} (2), 399--426, (2004). 

\end{thebibliography}
\end{document}